\documentclass[reqno,12pt]{amsart} 
\usepackage{amsmath,amssymb,amsfonts}
\usepackage{eucal}
\usepackage{enumerate}
\usepackage{graphicx}
\usepackage{cite}

 \usepackage{setspace}
 \usepackage{mathrsfs}
 \usepackage{algorithm}
 \usepackage[noend]{algpseudocode}
 \usepackage{xcolor}
  \usepackage{url}

\newcommand \comment[1]{}           
\renewcommand \comment[1]{\emph{[#1]}}      

\newtheorem{lemma}{Lemma}[section]
\newtheorem{corollary}[lemma]{Corollary}
\newtheorem{proposition}[lemma]{Proposition}
\newtheorem{theorem}[lemma]{Theorem}

\newtheorem{conjecture}[lemma]{Conjecture}

\theoremstyle{definition}

\newtheorem{example}{Example}[section]

\renewcommand{\phi}{\varphi}                 
\renewcommand{\epsilon}{\varepsilon}

\renewcommand\ell{l}

\newcommand\Tutte{\operatorname{Tutte}}
\newcommand\card{\operatorname{card}}

\newcommand\bbZ{\mathbb{Z}}

\newcommand\G{\Gamma}

\begin{document}

\pagestyle{myheadings} 
\markleft{\sc Rigoberto Fl\'orez and David Forge} 
\markright{\sc Activity from matroids to rooted trees and beyond } 
\thispagestyle{empty}

\title{Activity from matroids to rooted trees and beyond }

\author{Rigoberto Fl\'orez}
\address{The Citadel, Charleston, South Carolina 29409}
\email{\tt rigo.florez@citadel.edu}

\author{David Forge}
\email{\tt forge@lri.fr}

\date{\today}

\begin{abstract}
The interior and exterior activities of bases of a matroid are well-known notions that for instance permit one to define the Tutte polynomial. 
Recently, we have discovered correspondences between the regions of gainic hyperplane arrangements and colored 
labeled rooted trees.  Here we define a general activity theory that applies in particular to no-broken circuit (NBC) sets and labeled colored trees. 
The special case of activity \textsf{0} was our motivating case. As a consequence, in a gainic hyperplane arrangement the number  of bounded regions 
is equal to the number of the corresponding colored  labeled rooted trees of activity \textsf{0}.
\end{abstract}


 
\maketitle

\section{Introduction}\label{intro}

 An \emph{integral affinographic hyperplane} is a hyperplane of the form  $x_j-x_i=g$ in real affine space, where $g$ is an integer. 
 Many authors have been interested in this type of  hyperplane arrangement (see, for example,  \cite{Athanasiadis, PS, ForgeZaslavs}).  
 Some familiar examples are  (taking complete graphs) the \emph{braid} arrangement,  where $g=0$;  
 the \emph{graphic arrangements}, subarrangements of the braid arrangement;  the \emph{Shi} arrangement,  
 where $g=0$ or $g=1$ with $1\le i< j\le n$; the Catalan arrangement,  where $g=-1$, $0$, or $1$; 
 the \emph{Linial arrangement}, where $g=1$ with $1\le i< j\le n$ (denoted by $\mathscr{L}_n$). 
 The general type is called \emph{gainic arrangements}; they correspond with gain graphs.
 
 The basic definitions and the background given here are based on what is given more generally in \cite{ForgeZaslavs,ZaslavskyLectureGainG}.  
Let $\Gamma$ be a graph and take the group $\mathbb{Z}:=(\mathbb{Z}, +)$, where $-a$ is the inverse of $a$. 
Orient the edges of  $\Gamma$, and let $\vec{E}$ be the set of oriented edges. Denote by $-e$ the edge $e$ with its opposite orientation. 
An \emph{integral gain graph} $\Phi$ is a pair $(\G,\phi)$, where  $\phi: \vec{E} \to \bbZ$ satisfies  $\phi(-e)= -\phi(e)$ for all oriented edges $e$.  
The function $\phi$ is called the \emph{gain mapping} and $\phi (e)$ is the \emph{gain} of $e$.  
The gain of a walk $e_1e_2 \cdots e_l\;\text{ is }$
$$\phi(e_1e_2\cdots  e_l)=\phi(e_1)+\phi(e_2)+\cdots+\phi(e_l).$$ 
A circle $C$ is a $2$-regular connected graph, and it is \emph{balanced} if $\phi(C)=0$. For the sake of simplicity 
in this paper we use $V:=[n]=\{1,2, \dots, n\}$ as the set of vertices of 
$\Phi$ and $g(i,j)$ to represent an edge of $\Phi$ with orientation from $i$ to $j$. 
 We associate  $g(i,j)$ to the hyperplane $x_j- x_i = g$. 
 
Let $K_n$ be the complete graph and let $a\le b$ be integers. We denote by  $K_n^{[a,b]}$ the gain graph with vertex set $V=[n]$  and edges $k(i, j)$ with 
$a\le k\le b$ and $1\le i< j\le n$. Some familiar examples are the \emph{braid gain graph} $B_n:=K_n^{[0,0]}$, the  
\emph{Linial gain graph} $L_n:=K_n^{[1,1]}$, the \emph{Shi gain graph} $S_n:= K_n^{[0,1]}$, and the  \emph{Catalan gain graph} $C_n:= K_n^{[-1,1]}$. 

The gain graphs and affinographic hyperplane arrangements can be easily associated by matching the hyperplane with equation $x_j-x_i=g$ and the 
edge $(i,j)$ with gain $g$. Therefore, we will accept to say the hyperplane arrangement $K^{[a,b]}_n$ when it should be the hyperplane arrangement  
corresponding to the gain graph $K^{[a,b]}_n$.

Given a linear order $<_O$ on the set of edges $E$, a \emph{broken circuit} is the set of edges obtained by deleting the smallest element in a 
balanced circle. A set of edges, $N\subseteq E$, is a \emph{no-broken-circuit set},  denoted NBC, if it does not contain a broken circuit. This concept 
is from matroid theory (see for example \cite{Bjorner}).  It is well-known that this set depends on the choice of the order. However,
the cardinality of  the set of NBC sets of the gain graph does not depend on choice of an order.

 There are still many questions related to these families of hyperplane arrangements.  For instance, is it possible to find their  
 characteristic polynomials  and their number of regions? This may help us to answer other questions.  In fact, when the characteristic 
 polynomial is evaluated at $x=-1$ and $x=1$, it gives the number of regions and bounded regions (see, for example, Zaslavsky \cite{ZaslavskyFU}). 
 Once the number of regions is found, another good question could be, is it possible to find a bijective proof, and what is the 
 target set of a bijection? Of course these are not easy questions. For example, it is known that the regions of the Shi  arrangement are in correspondence 
 with parking functions (a target set) and the regions of the Linial arrangement are in correspondence with local binary search trees (a target set).  
 
 Let $b(\mathscr{L}_n)$ be the number of bounded regions of the Linial arrangement   $\mathscr{L}_n$: $x_i-x_j=1$ for $1\le i< j\le n$. 
 Athanasiadis \cite{Athanasiadis} found a closed  formula for $b(\mathscr{L}_n)$. He proved (based on generating functions) that 
\begin{equation}\label{AthanasiadisForm}
b(\mathscr{L}_n)=\frac{1}{2^n}\sum_{j=0}^{n} {n \choose j}(j-1)^{n-1} 
\end{equation}
and stated that ``it would be interesting to find a combinatorial interpretation for the numbers $b(\mathscr{L}_n)$". Here in 
Theorem \ref{activity1} we give a combinatorial interpretation for this formula. (Independently, Tewari \cite{Tewari} 
gave a different interpretation.)

The Tutte polynomial  is an important combinatorial invariant of a matroid that is more general than the characteristic polynomial of a 
hyperplane arrangement. There is an expansion of the Tutte polynomial using the bases and their interior and exterior activities. 
The no-broken-circuit (NBC) sets  of the arrangement have exterior activity  \textsf{0}. Thus, only  
interior activity applies to the NBC sets. The restriction of the Tutte polynomial to the  NBC bases gives an expansion of the activity 
polynomial (to be defined shortly) very close to that of the characteristic polynomial.

Some bijections from the NBC sets of gainic arrangements to different families of colored labeled rooted trees have been given by 
Forge et al.\ in  \cite{CFM,CFV, Forge}.  Bernardi \cite{Bernardi} found a  bijection between regions of the same hyperplane arrangements 
using the  trees as defined in \cite{CFM}. 
Levear  \cite{Levear} uses the same trees to find bijections for the $k$-dimensional faces of the same arrangements. 
 In this paper, we define an activity on those families of target trees that corresponds to the 
activity of the NBC bases. This will give a relation between the NBC bases of activity \textsf{0} and some special subsets of our target 
set. Since the NBC bases with activity \textsf{0} correspond to the bounded regions of the arrangement, this gives a  proof that the number of bounded  
regions is equal to the number of trees of activity \textsf{0}.

In section \ref{Sect2} we  give some background results which brought us to what will follow in the next sections. Specially Theorem \ref{partition}   
and Theorem \ref{TutteActivity} are the most important ones. 

In section \ref{ActivityPureCoveringSystem}, we define covering systems a generalization of matroids where both Theorem \ref{partition} and Theorem \ref{TutteActivity} are preserved. 
 A covering system $\mathcal{S}$ is a generalization of the concept of a matroid. We generalize activity of NBC bases  to an activity on a covering system $\mathcal{S}$. 
This activity gives rise to a polynomial, the activity polynomial. We show both some applications of the activity polynomial and its relations with  the 
Tutte polynomial. 
 
In Section \ref{Sect4}, we apply the results from Section \ref{ActivityPureCoveringSystem} to the set of rooted colored forests that was our motivation.  
Theorem \ref{activity1} gives that the number of  bounded regions of an affinographic arrangement is equal to the number of the corresponding  
colored forest with activity \textsf{0}.
 
 Finally we give some examples and constructive proofs for the number of NBC sets of a given activity for the braid arrangement, Shi arrangement, 
 and Lineal arrangement. The results in this paper give rise to a conjecture that we state at the end of the paper.  
 
\section{Some background and motivating results} \label{Sect2}

As a motivation for the study of NBC bases  of gain graphs, in this section we present two theorems that show relationships between the 
NBC bases of gain graphs and other areas of combinatorics. 

\subsection{Rooted colored trees} We color the edges of trees with the numbers $[1,k]$.  Let $T$ be a rooted $k$-colored tree.
The edges of $T$ are represented by $(u,w)$,  or by $uw$ (for simplicity)  if there is no ambiguity, where  $w$ is a child of $u$. 
For a fixed internal vertex  $v$  of $T$, we set  $c_v$ to mean the minimum color used by the edges going out of $v$.  
We say that $T$ is \emph{decreasing}  (\emph{increasing}) if for any internal vertex $v \in T$, the label of $v$ is  
larger (smaller) than the labels of each of its children $w_i$ such that $vw_i$ is colored with $c_v$. 
Note that the edges $vw_i$ colored with a bigger that $c_v$ are not considered. We say that $T$ is  
\emph{non-increasing} (\emph{non-decreasing})  if for any internal vertex $v \in T$, the label of $v$ is larger (smaller) than the label  
of at least of one of its children $w_i$ such that  $vw_i$ is colored with  $c_v$.   
Similarly, for $k_1+k_2=k$, we say that $T$ is $(k_1,k_2)$-colored decreasing if for any internal vertex $v$ such that  $c_v\le k_1$, the label of 
$v$ is larger than the label of each of its child $w_i$ and such that the edge $vw_i$ is colored with $c_v$. 
Note that in this definition the edges colored with the $k_2$ colors in $[k_1+1,k]$ never are  checked. Therefore, these $k_2$  
colors are  called free colors and the $k_1$ first colors are the non-free colors.   
The definitions of  $(k_1,k_2)$-colored increasing, $(k_1,k_2)$-colored non-increasing, and   
$(k_1,k_2)$-colored non-decreasing are similar and we omit them. 

\begin{example}[\bf decreasing tree using $4$ colors]
Figure \ref{figure1} Part (a) depicts a decreasing tree colored with $4$ colors.  In this figure we use $v_i$ to mean the vertex with label $i$. 
For example,  the vertex $v_{10}$, with label $10$, is the parent of $v_3$, $v_5$, and $v_7$. The minimum color used by 
$v_{10}v_{3}$, $v_{10}v_{5}$ and  $v_{10}v_{7}$ is $c_{v_{10}}=1$. We need to check decreasing condition only for the edge 
$v_{10}v_{3}$. From the figure can see that $v_{3}<v_{10}$, so we check it as OK.  
Similarly,  from the figure we can see that the edges $v_{3}v_{1}$ and $v_{3}v_{2}$ are colored with $1$ and the edges 
$v_{3}v_{4}$ and $v_{3}v_{11}$ are colored with $2$ and $3$, respectively. Therefore, $c_{v_{3}}=1$. Thus, we only need to check that in fact 
both satisfy that  $v_{1}<v_{3}$ and $v_{2}<v_{3}$. Finally, we observe that $c_{v_{7}}=2$. Since  $v_{6}<v_{7}$, we check it as OK. 
\end{example}

\begin{example}[\bf  $(2,2)$-colored decreasing tree] Figure \ref{figure1} Part (b) depicts a $(2,2)$ colored decreasing tree. In this example we have two non-free colors $1$  and  $2$, and two free colors  $3$ and $4$. The vertex $v_{6}$   
is the parent of $v_3$, $v_5$, and $v_{11}$. The colors used by the edges $v_6v_3$, $v_6v_5$, and $v_6v_{11}$ give that the minimum color is $c_{v_6}=\min\{1,2,3\}=1$. To check that the tree is 
$(2,2)$ colored decreasing at $v_6$, we need to check only the vertices  of the edges using the color $1$. So, in this case $v_{3}<v_{6}$. Therefore we check it as OK. 
Similarly, since $c_{v_3}=1$, we need to check that in fact these hold $v_{1}<v_{3}$ and $v_{2}<v_{3}$. Note that $c_{v_{5}}=3$, so there is nothing to check 
since $3$ is a free color. Finally, we observe that $c_{v_{11}}=2$, so we check that in fact the condition $v_7<v_{11}$ holds.   
\end{example}

\begin{figure} [htbp]
\begin{center} 
\includegraphics[scale=1.0]{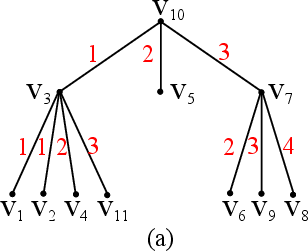} \hspace{2cm}
\includegraphics[scale=1.0]{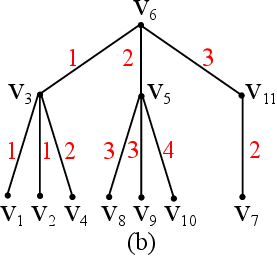}
\end{center}
\caption{ (a) Decreasing tree using $4$ colors \quad (b) $(2,2)$-colored decreasing tree. }
 \label{figure1}
\end{figure}

\subsection{Activity of NBC sets}\label{ActivityNBCSets}

Given a matroid $M$ on an ordered set $E$, an element $e$ is exteriorly (respectively, interiorly) active relative to a base $B$ if $e$ 
is the smallest  element of the  fundamental circuit $C(B,e)$ (respectively, cocircuit $C^*(B,e)$)  \cite{LasVergnas}. 
We use  $\iota (B)$ and $\epsilon (B)$ to denote the sets of interiorly and exteriorly active elements of a base $B$, and use $|\iota (B)|$ and 
$|\epsilon (B)|$ to denote the cardinality of those sets (these two cardinalities are called the \emph{interior} and \emph{exterior activity numbers}). These  
definitions were  introduced by Tutte for graphs and extended by Crapo \cite{Crapo} to matroids. The activity numbers give a formula for the Tutte polynomial:
\begin{equation}\label{tuttePoly}
\Tutte_A(x,y)=\sum_{B \textrm{ a basis}} x^{|\iota (B)|}y^{|\epsilon (B)|}.
\end{equation}

The \emph{NBC bases} are exactly those of exterior activity \textsf{0}. 
The activity of an NBC basis means only its interior activity.
Any subset of an NBC basis is called \emph{NBC set}. The following classic    
partition theorem is at the origin of our definition of a generalized activity. For simplicity, in this paper, when we refer to an activity number  
we mean the interior activity number. 

\begin{theorem}[\cite{Bjorner}]\label{partition}
Let $M$ be a matroid on a linearly ordered set $E$  and  let $\mathcal I$ be the set of all independent sets. Then: 

\begin{enumerate}
\item Every independent  set $I $ of $M$ can be uniquely written in the form $I = B \setminus Y$ for some  basis $B$ and some subset  
$Y \subset \iota(B)$. Equivalently, the intervals $[B \setminus \iota(B), B]$, form a partition of the set $\mathcal I$.

\item Every NBC set $A $ of $M$ can be uniquely written in the form $A = B \setminus Y$ for some NBC basis $B$ and some subset  
$Y \subset \iota(B)$. Equivalently, the intervals $[B \setminus \iota(B), B]$ form a partition of the set $\mathcal N$ of NBC sets.
\end{enumerate}

\end{theorem}

All these definitions are also valid for a semimatroid. An (affine) hyperplane arrangement defines the semilattice of flats of a semimatroid by 
considering all the non-empty  intersection sets of hyperplanes.  So, for simplicity, instead to say basis of the semimatroid defined by an arrangement 
we say basis of the arrangement. Similarly, whenever we need to say the basis of the semimatroid defined by NBC sets we say NBC basis and so on. 
For details on semimatroids see Ardila \cite{Ardila}. The following theorem is a classic result, see for example Crapo \cite{Crapo}, Tutte, and Zaslavsky's   
Theorem in the vocabulary of regions. 

\begin{theorem} \label{TutteActivity}  
The Tutte polynomial of an arrangement $A$ restricted to $y=0$ gives the activity polynomial, i.e., 
$\alpha(x)=\Tutte_A(x,0)=\sum_{i=0}^na_ix^i$. Furthermore, these hold:
\begin{enumerate}
\item \label{TutteActivityPart1} the activity polynomial is related to the  characteristic polynomial $\chi$ by 
$$\alpha(x)=(-1)^{n}\chi(1-x).$$

\item \label{TutteActivityPart2} The coefficient $a_i$ is equal to the number of NBC bases of activity $i$ of the arrangement.
\item \label{TutteActivityPart3} The value $\alpha(0)$ gives the number of bounded regions of the arrangements and the number of 
NBC bases of activity \textsf{0} of the arrangement.

\item \label{TutteActivityPart4} The value $\alpha(1)$ gives the number of NBC bases of the arrangement.
\item \label{TutteActivityPart5} The value $\alpha(2)$ gives the number of  regions and the number of NBC sets of the arrangement.

\end{enumerate}
\end{theorem}

\section{Activity of a covering system}\label{ActivityPureCoveringSystem}

We define an activity which is a generalization of the NBC activity using the partition property of Theorem \ref{partition}. 
The power set of a finite set  $E$ is denoted by $2^E$. An $r$-set is a set with $r$ elements, where the empty set is the $0$-set. 
Let $X\subseteq Y$ be finite sets;  we use $[X,Y]$ to denote the set $\{Z \mid X\subseteq Z\subseteq Y\}$,  called an \emph{interval}.   
An \emph{$r$-covering system}  $\mathcal{S}$ is an   ordered pair $(E, \mathcal{I})$ consisting of a finite set  $E$ and a collection 
$\mathcal{I}$ of subsets of $E$ of cardinality less than or equal to $r$, with the condition that  for every $I \in \mathcal{I}$ there is an 
$r$-set  $B \in \mathcal{I}$ such that $[I,B]\subseteq \mathcal I$.  An $r$-set  in $\mathcal I$  is called a \emph{basis} and the  set of all  
bases is  denoted by  $\mathcal B$. The elements of  $\mathcal{I}$ are called \emph{independent sets}. Note that we borrowed this terminology from matroids.  

For an $r$-covering system $\mathcal{S}$ the function $\emph{\texttt{a}}: \mathcal B \to 2^E$  is an \emph{activity}, if for every $B \in \mathcal{B}$  
it holds that  $\emph{\texttt{a}}(B)\subseteq B$ with  $[B\setminus \emph{\texttt{a}}(B),B]\subseteq \mathcal I$ and for every   
$I\in \mathcal I$   there is a unique $B \in \mathcal B$ such  that $B\setminus \emph{\texttt{a}}(B)\subseteq I\subseteq B$.  When these 
conditions are expressed by the  intervals of the form   $[B\setminus \emph{\texttt{a}}(B),B]$, they give rise to a partition of $\mathcal I$. 
The set $\emph{\texttt{a}}(B)$ and its  cardinality $|\emph{\texttt{a}}(B)|$  are called the set of \emph{active elements} of $B$ and the  
\emph{activity number} of $B$, respectively.  

The vector $A=(a_0,\ldots,a_r)$ is called  the   \emph{activity vector} of $\emph{\texttt{a}}$, where $a_i$ represents the number of sets 
of $\mathcal B$ of activity number $i$ (clearly,  $|\mathcal B|=\sum_{k=0}^{r} a_k$).  
The \emph{activity polynomial} is $\alpha(x)=\sum_{i=0} ^{r}a_i x^{i}$,  where $a_i$  
is an element of the activity vector. 

We say that  
$C=(c_0,\ldots,c_r)$ is the \emph{cardinality vector} of $\mathcal{S}$ 
if  $c_i$ is the number of sets in $\mathcal I$ of cardinality $i$. In particular, $c_r=|\mathcal B|$ and $c_0$ is equal to 1  if 
$\emptyset\in \mathcal I$  and $0$  otherwise. 

For example, the ordered pair  $\mathcal{S}:= (E,\mathcal I)$ is a covering system with three bases, where $E=\{1,2,3,4,5\}$ and 
\begin{multline*} \mathcal{I}=\{\{1,2,3,4\},\{1,2,4\},\{2,3,4\},\{2,4\},\{1,2,3,5\},  \{1,2,3\}, \{2,3,5\},\{2,3\},\{1,3,4,5\},\\
\{1,3,4\},\{1,3,5\},\{1,4,5\},\{1,3\},\{1,4\},\{1,5\},\{1\},\{2,3,4,5\}\}.
 \end{multline*}
An activity is given by 
\begin{eqnarray*}
\emph{\texttt{a}}(\{1,2,3,4\})=\{\textsf{1},\textsf{3}\},& \quad &\emph{\texttt{a}}(\{1, 2, 3, 5\})=\{\textsf{1},\textsf{5}\},\\
\emph{\texttt{a}}(\{1,3,4,5\})=\{\textsf{3},\textsf{4},\textsf{5}\},& \quad \text{ and } \quad &\emph{\texttt{a}}(\{2,3,4,5\})=\emptyset.
\end{eqnarray*}
In this case we have $A=(1,0,2,1,0)$,  $\alpha(x)=1+2x^2+x^3$ and $C=(0,1,5,7,4)$.
We show the intervals in the poset determined by the antichains of NBC bases in Figure \ref{Lattice2}. 

\begin{figure} [htbp]
\begin{center} 
\includegraphics[scale=.8]{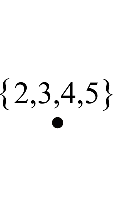} \hspace{1.5cm} \includegraphics[scale=.8]{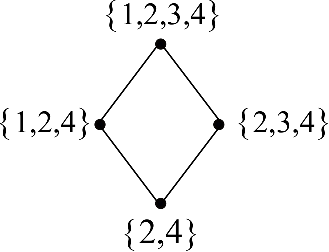}  \\ 
 \includegraphics[scale=.8]{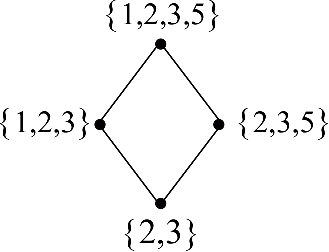} \hspace{2.5cm}  \includegraphics[scale=.8]{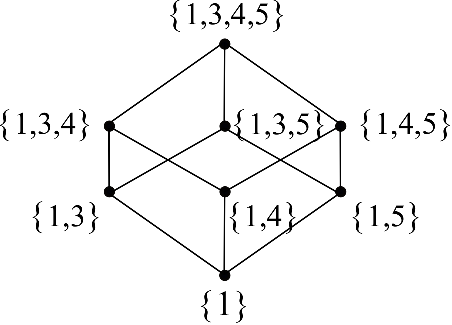}
\end{center}
\caption{Antichains of NBC bases.} \label{Lattice2}
\end{figure}

The following result is a natural generalization from matroid theory. So, its proof is straightforward.

\begin{proposition} \label{BasicOn Activities}
If $(E,\mathcal I)$ is a covering system with activity polynomial $\alpha$, then 
\begin {enumerate}
\item  \label{BasicOn ActivitiesPart1} the number of bases of activity \textsf{0} is given by  $\alpha(0)$,
\item  \label{BasicOn ActivitiesPart2} the number of bases is given by  $\alpha(1)$, and 
\item  \label{BasicOn ActivitiesPart3} the number of sets in $\mathcal I$ is given by $\alpha(2)$.
\end{enumerate}
\end{proposition}

\begin{proof} The proofs of Part  \eqref{BasicOn ActivitiesPart1} and Part \eqref{BasicOn ActivitiesPart2} are straightforward    
from the definition of $a_i$ (the number of bases of activity $i$). 

Proof of Part  \eqref{BasicOn ActivitiesPart3}. Since a base of activity $i$ covers exactly $2^i$ sets of $\mathcal I$, the conclusion  
follows from the partition property of  the activity.
\end{proof}

\begin{theorem}\label{Theorem:Act:Cardi:Vec}
Let $\mathcal{S}:= (E,\mathcal I)$ be an $r$-covering system with $C=(c_0,\ldots,c_r)$ its cardinality vector. 
Let $\textit{\texttt{a}}$ be an activity for $\mathcal{S}$ and let   
$A=(a_0,\ldots,a_r)$ be the corresponding activity vector. 
Then $A$ can be obtained  from $C$ by setting $a_{r}=c_{0}$ and  $a_{r-i}=c_{i} -\sum_{j=r-i+1}^{r} a_j {j\choose r-i}$.   
 In particular, in every case there is at most one basis of activity $r$.
\end{theorem}

\begin{proof} From the definition of activity  we know that a base $B$ of activity $\emph{\texttt{a}}(B)$ covers the interval 
$[B\setminus \emph{\texttt{a}}(B),B]$.  Therefore, a base of activity $\emph{\texttt{a}}(B)$ covers $2^{|\emph{\texttt{a}}(B)|}$   
independents, more precisely there are $ {|\emph{\texttt{a}}(B)| \choose r-i}$ independents of cardinality $i$.  Summing over  
all bases and using the fact that the intervals $[B\setminus \emph{\texttt{a}}(B),B]$ form a partition of  all independent sets,  
we get that the number of independent sets of  cardinality $i$ is  $\sum_{j=r-i+1}^{r} a_j {j\choose r- i}$.  This completes the proof.
\end{proof}

From Theorem \ref{Theorem:Act:Cardi:Vec} we obtain a relation between the activity polynomial $\alpha(x)$ and the cardinality 
polynomial $\card(x)=\sum c_i x^i$. We state it formally in this corollary. 

\begin{corollary}\label{TomCorollary} If $A=(a_0,\ldots,a_n)$ is an activity vector of $\textit{\texttt{a}}$, then 
$\card(x)=x^n \alpha({1/ x}+1)$.
\end{corollary}
 
The following result is a corollary of Corollary \ref{TomCorollary}\footnote{T. Zaslavsky suggested to call it  ``corollorollary''}. 

\begin{corollary}
All activities of a given $r$-covering system $\mathcal{S}$ have the same activity vector.
\end{corollary}

Two $r$-covering systems  $\mathcal{S}_1:=(E_1, \mathcal{I}_1)$  and $\mathcal{S}_2:=(E_2, \mathcal{I}_2)$,  with $|E_1|=|E_2|$,  
are \emph{similar } if  there is a bijection  $\psi $ from $\mathcal{I}_1$ to $\mathcal{I}_2$  such that for every $X \in \mathcal{I}_1$,    
$|\psi (X)|=|X|$.  Two similar covering systems have necessarily the same rank.

The following theorem makes the definition of activities works. Thus, the sizes of the various intervals are invariant. 
The proof of this theorem justifies defining polynomials that do not depend on the specific interval partition. Thus, it justifies the partition 
definition of activities.

\begin{theorem}\label{MainThmSimilar}  If  $\mathcal{S}$ and  $\mathcal{S}^{\prime}$ are similar $r$-covering systems with  activities 
$\texttt{a}$ and $\texttt{a}^{\prime}$, then their activity vectors  $A$ and $A^{\prime}$ are equal. 
\end{theorem}

\begin{proof}
From the definition of similarity, the number of sets of a given cardinality $k$ is the same in $\mathcal{S}$ and in $\mathcal{S}^{\prime}$, 
i.e., $c_k=c_k^{\prime}$ for every $1\le k\le n$. The conclusion follows from Theorem \ref{Theorem:Act:Cardi:Vec}.
\end{proof}

\begin{proposition}
Let $\mathcal{S}:=(E,\mathcal{I})$ be an $r$-covering system with bases $\mathcal B$ and with activity function $\texttt{a}$.   
For nonempty $\mathcal B^{\prime}\subseteq \mathcal B$  let 
$\mathcal I^{\prime}=\{I\in\mathcal{I} \mid \exists B\in \mathcal{B}^{\prime} \text{ such that } B\setminus \texttt{a}(B)\subseteq I \subseteq B\}$.  
Then $\mathcal{S}^{\prime}:=(E,\mathcal I^{\prime})$ is an $r$-covering system and $\texttt{a}^{\prime}=\texttt{a}|_{\mathcal B^{\prime}}$ is an activity. 
\end{proposition}

\begin{proof}
The intervals $[B\setminus \emph{\texttt{a}}^{\prime}(B),B] = [B\setminus \emph{\texttt{a}}(B),B]$ for $B \in \mathcal{B}^{\prime}$ form a partition of $\mathcal{I}^{\prime}$.
\end{proof}

\begin{proposition}\label{UnionCoveringSystem}
Let $\mathcal{S}_{1}:=(E,\mathcal{I}_{1})$ and  $\mathcal{S}_{2}:=(E,\mathcal{I}_{2})$ be $r$-covering systems with  activities  
$\texttt{a}_{1}$ and $\texttt{a}_{2}$. If  $\mathcal{I}_{1} \cap \mathcal{I}_{2}=\emptyset$, then 
$\mathcal{S}:=(E,\mathcal{I}_{1} \cup \mathcal{I}_{2})$ is a covering system with activity defined by 
 \[
\texttt{a}(B)=
\begin{cases}
\texttt{a}_{1}(B), & \text{ if } B \in \mathcal{I}_{1}, \\
\texttt{a}_{2}(B), & \text{ if } B \in \mathcal{I}_{2}.
\end{cases}
\]
\end{proposition}

\begin{proof}
The intervals $[B\setminus \emph{\texttt{a}}_{1}(B),B]$ for $B \in \mathcal{B}_{1}$ together with the intervals 
$[B\setminus \emph{\texttt{a}}_{2}(B),B]$ for $B \in \mathcal{B}_{2}$ form a partition of $\mathcal{I}_{1}\cup\mathcal{I}_{2}$.
\end{proof}

Let $\mathcal{S}:=(E,\mathcal{I})$  be an $r$-covering system with an activity  $\textit{\texttt{a}}$ and let  
$X$ be an element of $\mathcal{I}$. We define $\mathcal{I}_{X}:=\{I \in \mathcal{I}: X \subseteq I \}$, 
$\mathcal{B}_{X}:=\{X \in \mathcal{B}: X \subseteq B \}$, and $\textit{\texttt{a}}_{X}(B):=\textit{\texttt{a}}(B)\setminus X$, for every $B\in\mathcal{B}_{X}$.

\begin{proposition}\label{XCoveringSystem}
Let $\mathcal{S}:=(E,\mathcal{I})$  be an $r$-covering system with an activity  $\textit{\texttt{a}}$ and let  
$X \in \mathcal{I}$. Then
$\mathcal{S}_X:=(E,\mathcal{I}_{X})$ is an $r$-covering system and $\texttt{a}_{X}$ is an activity of $\mathcal{S}_X$.
\end{proposition}

\begin{proof}
The set $\mathcal{I}_{X}$  is the intersection of $\mathcal{I}$ and the interval $[X,E]$. Therefore, every interval  
$[B\setminus \emph{\texttt{a}}(B),B]$ in the partition of $\mathcal{I}$ that intersects $[X,E]$ gives the interval  
$[(B\setminus\emph{\texttt{a}}(B))\cup X,B]$, if $X\subseteq B$. (Note that  
$(B\setminus\emph{\texttt{a}}(B))\cup X=B\setminus ( \emph{\texttt{a}}(B)\setminus X)$ if $X\subseteq B$.)  They form a partition of $\mathcal{I}_{X}$.
\end{proof}

\subsection{ Examples of covering systems with activities}

\begin{example}[\bf \emph{$r$-pure independence system}]

Let $E$ be a finite set. We say that a collection $\mathcal{I}$ of subsets of $E$ is an \emph{$r$-pure independence system} if   
every maximal element in  $\mathcal{I}$ has cardinality $r$ and if for each $I \in \mathcal{I}$, every subset of $I$ is also in $\mathcal{I}$.   
Clearly, every $r$-pure independence system is an $r$-covering system.  
But an $r$-pure independence system does not necessarily have an activity. For example, let $(E,\mathcal{I})$ be a $2$-pure independence  
system with $E=\{1,2,3,4\}$, where $\mathcal{I}$ has exactly two maximal elements, $B_1=\{1,2\}$ and $B_2=\{3,4\}$. We claim that 
$(E,\mathcal{I})$  does not have an activity. Indeed, suppose that $\emph{\texttt{a}}$ is such an activity; then to ``cover" the empty 
set we need $\emph{\texttt{a}}(B_i)=B_i$  for either $i=1$ or $i=2$. Let us suppose that  $\emph{\texttt{a}}(\{1,2\})=\{\textsf{1},\textsf{2}\}$  
(the other case is similar); to cover $\{3\}$ we need  $\emph{\texttt{a}}(\{3,4\})=\{ \textsf{4}\}$, but then we do not have any possibility  
to cover $\{4\}$, which is a contradiction. 

Another way to prove that there is not an activity is by Theorem \ref{Theorem:Act:Cardi:Vec}.  
Here the  cardinality vector of  $(E, \mathcal{I})$ is given by $C=(1,4,2)$. From Theorem \ref{Theorem:Act:Cardi:Vec} 
we know that if there is an activity vector it is  $A=(-\textsf{1},\textsf{2},\textsf{1})$. Since the first entry is negative, there is no such activity. 
\end{example}

\begin{example}[\bf matroids]
The set of independent sets of a matroid and the set of NBC sets of a matroid are examples of $r$-pure independence systems.   
The activity defined by the order of the elements as explained in Section \ref{ActivityNBCSets} is an activity for both of them by Theorem \ref{partition}.

For the set of NBC sets the activity polynomial $\alpha_{0}$ is related to the Tutte polynomial by $\alpha_{0}(x)=\Tutte(x,0)$. So,
$\alpha_{0}(1)=\Tutte(1,0)$ is the number of NBC bases and $\alpha_{0}(2)=\Tutte(2,0)$ is the number of NBC sets.  
The number $\alpha_{0}(2)$ is, at the same time, the number of regions of the arrangement when the matroid is 
 defined by an arrangement of hyperplanes.  

For the set of independent sets the activity polynomial $\alpha_{1}$ is related to the Tutte polynomial by $\alpha_{1}(x)=\Tutte(x,1)$. So, 
$\alpha_{1}(1)=\Tutte(1,1)$ is the number of bases and $\alpha_{1}(2)=\Tutte (2,1)$ is the number of  independents sets of the matroid.
\end{example}

\begin{example}[\bf colored rooted forests] 

Let $E$ be the set of all directed edges with vertices in the set  $V=\{1,2, \dots, n\}$. Let $\mathcal{I}$ be the set of rooted forests on the set $V$.  
The rooted forests are such that every vertex has at most one incoming   edge; the roots are the vertices with indegree zero. 
In a rooted forest, each component is a directed tree and has exactly 
one root. The set $B$ of edges   of a connected  acyclic graph with $n$ vertices is defined as a base in $\mathcal{B}$. The couple  
$(E, \mathcal{I})$ is an $(n-1)$-covering system on the set of ordered edges $\{(i,j): 1\le i \ne j \le n\}$.  
It is a pure independence system. 
The subset of decreasing rooted forests  is still a pure independence system but the set of non-increasing rooted  
forests is not. Indeed, there are some edges that cannot be deleted while keeping the non-increasing property.  For the example, 
the tree defined with a root $r=2$ and edges $(2,1)$ and $(2,3)$ is non-increasing tree, and  if we remove the edge $(2,1)$  
(but not its vertices) we obtain the forest formed by the vertex $1$ and the edge $(2,3)$. So, it does not give a non-increasing forest.  
Therefore, the set of non-increasing forest is not a pure independent system. Nevertheless, it is still a  
covering system, since it is possible to add some edges to any non-increasing trees. A potential activity is the set of edges $(n,v)$ in a tree.
This will be discussed later, in the next section. 
\end{example}

\begin{example}[\bf paths connecting two vertices]
We now give an application of Propositions \ref{UnionCoveringSystem} and \ref{XCoveringSystem}.  Let $G=(V,E)$ be a connected graph, where  
$|V|=n$, and let $a$ and $b$ be distinct vertices. Let $\mathcal{F}_{ab}$  be the set  all forests containing a path connecting $a$ and $b$.  
This set forms an $(n-1)$-covering system. Let  $\mathcal{P}=\{P: P \text{ is a path connecting } a \text{ and } b \}$ and for a fixed path $P$  
connecting $a$ and $b$, let $\mathcal{F}_{P}$ the set of all forests containing $P$.  We have 
$\mathcal{F}_{ab} =\uplus_{P\in \mathcal{P}} \mathcal{F}_{P}$ ($\uplus$ means disjoint union). 
From Propositions \ref{UnionCoveringSystem} and \ref{XCoveringSystem} we see that $\mathcal{F}_{ab}$ is an $(n-1)$-covering system and it  
also has an activity. Note that in this example the empty set is not an independent set. 
\end{example}

\section{activity of colored rooted trees}\label{Sect4}

Corteel et al. \cite{CFM} define some classes of colored rooted trees and show that they correspond to the NBC sets of the  
deformations of the braid arrangement with an interval $[a,b]$ containing 0 or 1.  The original motivating examples were the  
decreasing labeled trees and the local binary search trees. Following Stanley \cite{Stanley}, a   
\emph{local binary search tree} ---LBS tree, for short--- is a labeled rooted binary tree, with vertex set $[n]$, such that every 
left child of a vertex is less than its parent, and every right child is greater than its parent. Note that a parent may have only 
one child. The number of LBS trees on the set $[n]$ is known  to be equal to the number of regions of the Linial   
arrangement in dimension $n$; see \cite{PS}.

In general, if $T$ is $(k_1,k_2)$-decreasing (increasing), then $T$ is $(k_1-\alpha,k_2+\alpha)$-decreasing (increasing) if $k_1-\alpha\ge 0$. 

The following theorem shows a  relation between the number of NBC bases  
and the total number of decreasing (non-increasing)  colored trees.  
Part \eqref{CFMPart1}  corresponds to the Shi arrangement with interval $[0,1]$ and to the generalizations $[-k_2,1+k_2]$, 
Part \eqref{CFMPart2}  corresponds to the so called Catalan arrangements $[-k_2,k_2]$; 
and to some generalization $[-k_2,k_1+k_2]$; 
Part \eqref{CFMPart3} corresponds to the Linial arrangement $[1,1]$ and to some generalizations 
$[1,k_1]$, and  $[1-k_1,k_1+k_2]$.

The proofs of  Theorem \eqref{CFM} Parts \eqref{CFMPart1} and \eqref{CFMPart2} are in  \cite{CFV} and the proof of  Part
\eqref{CFMPart3} is in \cite{Forge}.

\begin{theorem}\label{CFM} Let $k_1$ and $k_2\ge 0$. If $K_n^{[a,b]}$ is the gain graph with vertices $V=[n]$, then:
\begin{enumerate}
\item \label{CFMPart1} there is a bijection between the set of NBC sets of $K_n^{[-k_2,1+k_2]}$ and the set of forests with $k_2$ free colors.

\item \label{CFMPart2}   There is a bijection between the set of NBC sets of $K_n^{[-k_2,k_1+k_2]}$ and the set of $(k_1+1,k_2)$-decreasing colored forests. 

\item \label{CFMPart3} There is a bijection between the set of NBC sets of $K_n^{[1-k_2,k_1+1+k_2]}$ and the sets of $(k_1+1,k_2)$-non-increasing colored forests.

\end{enumerate}
\end{theorem}

 This theorem is the main research object by Corteel et al. in \cite{CFM}. From Theorem \ref{CFM} we can deduce many unexpected 
 correspondences. For example, there is a bijection between the NBC sets of $K_n^{[0,2]}$ and both the $(3,0)$ decreasing trees 
 (setting $k_1=2$ and $k_2=0$ in Part \ref{CFMPart2}) and the $(1,1)$ non-increasing trees (setting $k_1=1$ and $k_2=1$ in  
 Part \ref{CFMPart3}). In \cite{CFM} the authors gave  a bijective proof in addition to general proof. 

We recall that given an edge $e$ in a tree, we distinguish one vertex of $e$ as the parent of the second vertex (called child)   
if it is closer to the root. Let $k$ be a positive integer with $k=k_1+k_2$, where $k_1$ and $k_2 \ge 0$ and let $T$ be a tree with vertex 
set equal to $[n]$. If $T$ is either a $(k_1,k_2)$-decreasing  $k$-colored tree or is a $(k_1,k_2)$-non-increasing $k$-colored tree,   
then we define 
$\emph{\texttt{a}}(T)$ as the set of edges having $n$ as their parent and colored with $1$.

\begin{proposition}\label{ActivityP42} Let $k$ be a positive integer with $k=k_1+k_2$, where $k_1$ and $k_2 \ge 0$. If $T$ is either a  
$(k_1,k_2)$-decreasing  $k$-colored tree or is a $(k_1,k_2)$-non-increasing $k$-colored tree, then $\texttt{a}(T)$ is an activity. 
\end{proposition}

\begin{proof} We prove the case in which a forest $F$ is $k$-colored and $(k_1,k_2)$-decreasing with vertex set $[n]$. The other case is similar and we omit it. 

We want to show that the sets $[T \setminus \emph{\texttt{a}}(T), T]$ form a partition of the set of forests. For any forest $F$ with $d$ 
components $T_1, T_2, \dots, T_d$ not containing the vertex $n$. (There is a $(d+1)$th component $T_0$ containing the vertex $n$.)
The unique tree $T$ such that $F \in [T \setminus \emph{\texttt{a}}(T), T]$ is the tree obtained by adding the $d$ edges $(n,r_i)$ to $F$,  
for $1\le i\le d$ where $r_i$ is the root of the tree $T_i$.
\end{proof}

Our motivation for the rest of this paper is to provide a bijective proof of the formula of Athanasiadis (see equation \eqref{AthanasiadisForm}). 
He  wanted an interpretation similar to the one given by \cite[Theorem 4.1]{StanleyHyperplane} for the number of regions $b(\mathscr{L}_n)$.   
The following theorem answers that question.  

For the following theorem we use $K^{[a,b]}_n$ for the general case. But for each part we give the details how $a$ and $b$ will be. 

\begin{theorem}\label{activity1} Let $k_1$ and $k_2 \ge 0$ and $0\le \alpha \le n-1$. Then 

\begin{enumerate}
\item \label{activity1Part1} The number of NBC bases of $K^{[-k_2,1+k_2]}_n$ with activity $\alpha$ is equal to the number of trees with $k_2$ colors having $\alpha$ edges 
of the form $(n,i)$ colored with $1$. In Particular, the number of bounded regions of the arrangement $K^{[-k_2,1+k_2]}_n$ is equal to the number of trees with  $k_2$ colors having no-edge of the form $(n,i)$ colored with $1$.  

\item  \label{activity1Part2} The number of NBC bases of $K^{[-k_2,k_1+k_2]}_n$ with activity $\alpha$ 
is equal to the number of $(k_1+1,k_2)$ colored decreasing trees with $\alpha$ edges of the form $(n,i)$ colored with $1$. In particular, the number of 
bounded regions of the arrangement $K^{[-k_2,k_1+k_2]}_n$ is equal the number of $(k_1+1,k_2)$ decreasing trees with no-edge on the form 
$(n,1)$ colored with $1$.

\item  \label{activity1Part3} The number of NBC bases of $K^{[1-k_2,k_1+k_2]}_n$ of activity $\alpha$ is equal to the number of $(k_1,k_2)$ colored non-increasing tress with 
$\alpha$ edges  of the form $(n,i)$ colored with $1$. In particular the number of bounded regions of $K^{[1-k_2,k_1+k_2]}_n$ is equal to the number of 
$(k_1,k_2)$  colored non-increasing trees with no-edge of the form $(n,i)$ colored with $1$.
\end{enumerate}

\end{theorem} 

\begin{proof} We prove Part \eqref{activity1Part1}.  
From Theorem \ref{CFM} Part \eqref{CFMPart1} we have a bijection between the set of NBC sets of  $K^{[-k_2,1+k_2]}_n$ and the set of forests 
with $k_2$ colors. Since these two sets are covering systems, they are similar covering systems. From  Proposition \ref{ActivityP42}   
we know that for the sets of trees with $k_2$ colors, the number of edges of the form $(n,i)$ colored with $1$, is an activity. Therefore,  
Theorem \ref{MainThmSimilar} imply that their activity vectors are equal.

We now prove the bounded regions part.  It follows, just recalling that the number of bounded NBC bases of activity $\textsf{0}$ of a gain  
graph is equal to the number of bounded regions of the corresponding hyperplane arrangements (see for example,   
Theorem \ref{TutteActivity} Part \eqref{TutteActivityPart3}).

Note that the proofs of Parts \eqref{activity1Part2} and \eqref{activity1Part3} are identical ---replace in above argument   
Theorem \ref{CFM} Part \eqref{CFMPart1} by Theorem \ref{CFM} Part \eqref{CFMPart2} and  \eqref{CFMPart3}, respectively.
\end{proof}

\begin{example}[\bf activity numbers of $\mathscr{L}_3$]
In Figure \ref{Trees2} we represent all non-increasing trees with three vertices. From Theorem \ref{CFM} we  
know that they correspond to the NBC bases of the gain graph  $K^{[1,1]}_n$ corresponding to the Linial arrangement.

The activity number of a base is simply the number of children of  
vertex number 3.  We do not color the edges of the trees in Figures \ref{Trees2} and \ref{Trees3} because there is only one color (the color $1$). 
The numbers are, from left to right, \textsf{2}, \textsf{1} and \textsf{0}. It is easy to see that  $\alpha (x)=x^2+x+1$. Note that 
$$
b(\mathscr{L}_3)=(1/8)\left({3 \choose 0}\cdot (-1)^2 +{3 \choose 2}\cdot 1^2+{3 \choose 3}\cdot 2^2\right)=1.
$$
\end{example}

\begin{example} [\bf activity numbers of $\mathscr{L}_4$]
In Figure \ref{Trees3}, we represent all non-increasing  trees with four vertices. By Theorem \ref{CFM} we know that they correspond to the 
NBC bases of  the Linial arrangement. The activity number of a base is simply the number of children of vertex number 4. There are one tree  
of activity number \textsf{3},  three trees of  activity number \textsf{2}, six trees of activity number \textsf{1} and four  tress of activity number 
\textsf{0}.  Therefore,  $\alpha (x)=x^3+3x^2+6x+4$. Finally, we note that there are four trees of activity \textsf{0}. That is, there are four bounded 
regions. We also note, from the equation \eqref{AthanasiadisForm}, that we have 
$$
b(\mathscr{L}_4)=(1/16)\left({4 \choose 0}\cdot (-1)^3 +{4 \choose 2}\cdot 1^3+{4 \choose 3}\cdot 2^3+{4 \choose 4}\cdot 3^3\right)=4.
$$
\end{example}

\begin{figure} [htbp]
\begin{center} 
\includegraphics[scale=.8]{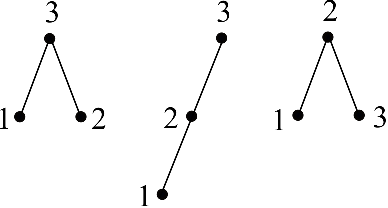} \hspace{2cm}
\end{center}
\caption{The activities are, in order from left to right, \textsf{2}, \textsf{1} and \textsf{0}.} \label{Trees2}
\end{figure}

\begin{figure} [htbp]
\begin{center} 
\includegraphics[scale=.8]{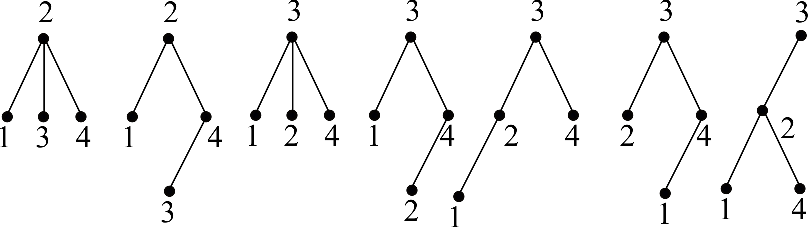}\\[14pt]
\includegraphics[scale=.8]{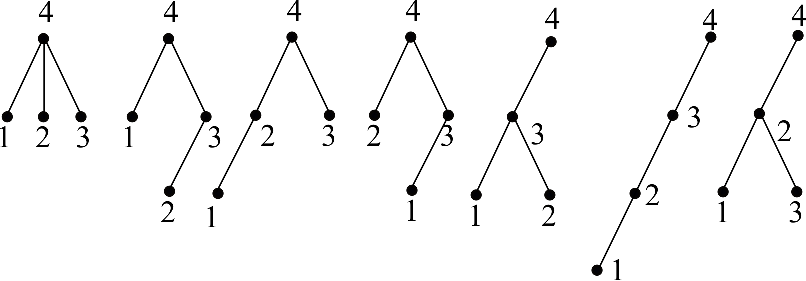}
\end{center}
\caption{ The activities are, from left to right, \textsf{0}, \textsf{1}, \textsf{0}, \textsf{1}, \textsf{0}, \textsf{1}, 
\textsf{0}, \textsf{3}, \textsf{2}, \textsf{2}, \textsf{2}, \textsf{1}, \textsf{1}, \textsf{1}.} \label{Trees3}
\end{figure}

\section{The loosing activity}

We started this research with the aim of finding something in the binary search trees that would correspond to activity in the NBC trees. 
We wanted to use NBC trees of the gain graphs as done by Forge in \cite{Forge}. 
After some experimentation, we figure out that the number of edges of the form $(i+1,i)$ were this something that we where looking for. 
We thought to call it ``activity". But now it is only the  loosing activity, since it does not work as nicely  as the edges of the form $(n,i)$. 
Nevertheless, we manage to prove some interesting results about the number of trees with $k$ edges of the form $(i+1,i)$.  
We conjecture a similar result in the case of LBS trees.  

In the previous section, we gave a general definition for an activity. We showed that the set of edges of the form $(n,i)$ is an activity for 
the general rooted trees (Shi), the decreasing trees (braid), and the non-decreasing trees. We were searching to find out if the number of edges   
of the form $(i+1,i)$ was the correct activity. However, we realized that the correct one was the number of edges of the form $(n,i)$.  
Here we give some results obtained from the original question. So, the aim of this section is to give a count (equivalent to an activity)   
for rooted  labeled trees that fits the activity of the corresponding NBC bases. With this objective in mind, we start with some classical 
arrangements ---the braid arrangement, the Shi arrangement, and the Linial arrangement. We analyze these arrangements with rooted 
trees ---decreasing trees, general trees, and local binary search trees (LBS trees).  

We knew the  activity  numbers for every $n$, so we could try different possible definitions. 
By looking at the first cases, we found  two different candidates for activity for trees on $n$ vertices:
\begin{enumerate}
\item the edges $(n,i)$ in the tree, 
\item the edges $(i+1,i)$ in the tree.
\end{enumerate}
In this section we prove that these two activities give the same numbers for these two cases of increasing trees and of general trees. 

\subsection{The braid case} The NBC set  in $B_n$, the braid arrangement in $n$ dimensions,  corresponds to the set of increasing labeled trees. 
Internally active edges in such a tree are just the edges $(n,i)$. 
We recall that the rising factorial is defined as $x^{(n)}=(x+1)(x+2)\cdots (x+n-1)$. 
The coefficient of $x^k$ in $x^{(n)}$ is an unsigned Stirling number of the first kind denoted by $(-1)^{n-k}s(n,k)$. 

\begin{theorem} \label{NBAinBn}
The number of NBC sets of activity $k$ in $B_n$ is equal to these:
\begin{enumerate}
\item \label{NBAinBnPart0}  the number of increasing trees on $n$ vertices where the vertex $1$ has degree $k$.
\item \label{NBAinBnPart1}  The number of decreasing trees on $n$ vertices where the vertex $n$ has degree $k$.
\item \label{NBAinBnPart2} The number of decreasing trees on $n$ vertices with $k$ edges of the form $(i+1,i)$.
\item \label{NBAinBnPart3} $(-1)^{n-k}s(n,k)$.
\end{enumerate}
\end{theorem}

\begin{proof} First of all, we observe that the number of NBC sets of activity $k$ in $B_n$ is equal to the Tutte polynomial given in  Equation 
\eqref{NBAinBnPart0} on Page \pageref{NBAinBnPart0}. 
This follows by Theorem \ref{CFM} Part \eqref{CFMPart1} with $k_2=0$ and Theorem \ref{activity1}.

It is straightforward to see that Parts \eqref{NBAinBnPart0} and  \eqref{NBAinBnPart1} are equivalent. Note that Part \eqref{NBAinBnPart1} is   
equivalent to Part \eqref{NBAinBnPart3} (it follows from Theorems \ref{CFM} Part \eqref{CFMPart1} and from Theorem \ref{activity1}).

We now prove that Part \eqref{NBAinBnPart1} and Part \eqref{NBAinBnPart2} are equivalent.  We just define an involution on the set of decreasing  
 trees on $n$ vertices that send a tree with with $k$ edges of the form $(n,i)$ to a tree with $k$ edges of the form $(i+1,i)$.
 For a tree $T$, we just need to replace every edge in $T$ of the form $(n,i)$ by the edge $(i+1,i)$ and every edge in $T$ of the form $(i+1,i)$ by the edge 
 $(n,i)$. This completes the proof.
 \end{proof}

Note that $n$ is the root of any decreasing tree and that therefore there are no decreasing tree of activity $\textsf{0}$.
This shows that the braid arrangement has 
no bounded regions. We also note that both the sets of edges $(n,i)$ and of edges $(i+1,i)$ define an activity in the system of increasing trees. 
This gives a second proof that Part \eqref{NBAinBnPart1} and Part \eqref{NBAinBnPart2} are equivalent.

\subsection{The Shi case}

The NBC bases of the Shi arrangement correspond to the rooted trees. The aim of Theorem \ref{NumberNBCTrees} is to show that $(n,i)$ or  
the edges $(i+1,i)$ give rise to the same activities as the NBC bases of the Shi arrangement. Part \eqref{NumberNBCTreesPart1} is a  
special case of  Theorems \ref{CFM} and \ref{activity1}. However,  Part \eqref{NumberNBCTreesPart2} is not a special case, because the  
edges $(i+1,i)$ do not define an activity anymore. For the proof of the following theorem see the Appendix.  

\begin{theorem} \label{NumberNBCTrees}
The number of NBC sets of activity $k$ in the Shi arrangement $S_n$ is equal to these:
\begin{enumerate}
\item \label{NumberNBCTreesPart1} the number of rooted trees on $n$ vertices where the vertex $n$ has degree $k$.
\item\label{NumberNBCTreesPart2}  The number of rooted trees on $n$ vertices with $k$ edges of the form $(i+1,i)$.
\item  \label{NumberNBCTreesPart3}  $ {n-1\choose k} (n-1)^{n-1-k}.$
\end{enumerate}
\end{theorem}

\subsection{ Examples} The examples in this section are based on the algorithms from the Appendix. Consider the tree given in Figure \ref{TreesBlue}.

\begin{example}[\bf Pr\"uffer decodings] We decode $w_p=265256$ with $n=7$. 

 \textbf{Step $i$=1}. Find the first letter in $w_p$ after the last $1$, and then replace it in $w_p$ by $(2,1)$. Set $w_p:=(2,1)65256$. 

   \textbf{Step $i$=2}. Find the first letter in $w_p$ after the last $2$, and then replace it in $w_p$ by $(5,2)$. Set $w_p:=(2,1)652(5,2)6$.

  \textbf{Step $i$=3}. Find the first letter in $w_p$ after the last $3$, and then replace it in $w_p$ by $(6,3)$. Set $w_p:=(2,1)(6,3)52(5,2)6$.

  \textbf{Step $i$=4}. Find the first letter in $w_p$ after the last $4$, and then replace it in $w_p$ by $(5,4)$. Set $w_p:=(2,1)(6,3)(5,4)2(5,2)6$.

  \textbf{Step $i$=5}. Find the first letter in $w_p$ after the last $5$, and then replace it in $w_p$ by $(6,5)$. Set $w_p:=(2,1) (6,3) (5,4) 2 (5,2) (6,5)$.

  \textbf{Step $i$=6}. The root is equal to $6$.

  \textbf{Step $i$=7}. Find the first letter in $w_p$ after the last $6$, and then replace it in $w_p$ by $(2,7)$. Set $w_p:=(2,1) (6,3) (5,4) (2,2) (5,2) (6,5)$.

  \textbf{Output}. The tree has edges $(2,1)$, $(6,3)$, $(5,4)$, $(2,7)$, $(5,2)$, $(6,5)$. See Figure \ref{TreesBlue}.

\end{example}

\begin{example}[\bf Blue coding and decoding]
This coding is given by  $w=b5b24b$ with $n=7$ with output depicted in Figure \ref{TreesBlue}.

The decoding is: 

 \textbf{Step $i$=1}. The parent of $1$ is given by the first letter (left-to-right), in this case it is $b$. Replace $b$ by $(2,1)$, as given in the tree in 
Figure \ref{TreesBlue},  and  set $w=(2,1)5b24b$.

  \textbf{Step $i$=2}.  The parent of $2$ is given by the first letter after the last $2$, in this case it is $4$. Since $4>2$, from the tree in 
Figure \ref{TreesBlue} we can see that $4$ must be replaced by $(5,2)$.
Now  set $w=(2,1)5b2(5,2)b$.

  \textbf{Step $i$=3}.  The parent of $3$ is given by the first regular letter,  in this case it is $5$. Since $5>3$, from the tree in 
Figure \ref{TreesBlue} we can see that $5$ must be replaced by $(6,3)$. 
Now  set  $w=(2,1)(6,3)b2(5,2)b$.

  \textbf{Step $i$=4}.  The parent of $4$ is given by the first regular letter, in this case  it is $b$. Replace $b$ by $(5,4)$ and  set $w=(2,1)(6,3)(5,4)2(5,2)b$.

  \textbf{Step $i$=5}.  The parent of $5$ is given by the first letter after the last $5$, in this case it is $b$. Replace $b$ by $(6,5)$ and set 
$w=(2,1)(6,3)(5,4)2(5,2)(6,5)$.

Note that $6$ is the last letter in the last outcome of $w$, so we conclude that $6$ is the root.

  \textbf{Step $i$=6}.  The parent of the vertex $7$ is the vertex $2$. Then replace $2$ by $(2,7)$ and set $w=(2,1)(6,3)(5,4)(2,7)(5,2)(6,5)$. 

  \textbf{Output}. The tree has edges $(2,1)$, $(6,3)$, $(5,4)$, $(2,7)$, $(5,2)$, $(6,5)$. See, Figure \ref{TreesBlue}.

\begin{figure} [htbp]
\begin{center} 
\includegraphics[scale=.8]{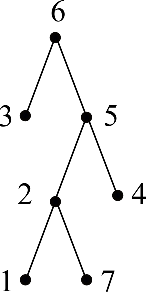} 
\end{center}
\caption{Tree with root 6, and edges $(2,1)$, $(6,3)$, $(5,4)$, $(2,7)$, $(5,2)$, $(6,5)$.} \label{TreesBlue}
\end{figure}

We can check that the coding of this tree is indeed the starting word.
\end{example}

\section{Some conclusions and remarks}

A \emph{local binary search tree} (LBS tree, for short) is a planar rooted  tree for which every vertex $v$ has two possible children, the left one denoted by  
$l$ and the right one denoted by $r$, ordered so that $l<v$ (in case that it exists) and $r>v$ (in case it exists). A left LBS tree is a LBS tree whose root has no 
right child. We believe that the number of left LBS trees with $k$ edges of the form $(i+1,i)$ (recall that the vertices are $[n]=\{1,2,\ldots,n\}$)  
is equal to the number of  left LBS trees with $k$ edges of the   form $(n,i)$.  This is stated formally  in Conjecture \ref{conjecture1}. 

Motivated by the Athanasiadis' question \cite{Athanasiadis}, our main interest from the beginning of this project was the Linial case. The number of regions 
of the Linial  arrangement is known to be equal to the number of LBS trees. 
So, we were expecting to give some kind of a activity  on the edges of the LBS trees, which match with the activity of NBC trees in $K^{[1,1]}_n$. 
We were looking only at the left LBS trees which correspond to NBC trees as all LBS trees correspond to all NBC sets. At some point we thought  
that in left LBS trees the number of edges of the form $(i+1,i)$ should be the desired ``activity". A computer exploration  give some confirmation on 
this belief.  So, we left it as a conjecture (see Conjecture \ref{conjecture1}). 
 
Let us first see that the left LBS trees are actually equivalent to  non-increasing rooted forests. That correspondence is directly described   by a  
\emph{rotation} of the tree.  A \emph{local binary search tree} (LBS tree, for short) is a planar rooted  tree for which every vertex $v$ has two  
possible children, the left one denoted by   $l$ and right one denoted by $r$, such that $l<v$ (in case it exists) and $r>v$ (in case it exists).  
A left LBS tree is an LBS tree whose root has no right child. The correspondence holds by replacing all right edges of an LBS tree as follows; 
take a right edge $(v,r)$ and find the first left ancestor  $x$ of $v$,  and then replace $(v,r)$ by $(x,r)$. 

In Figure \ref{CorrepondenceLBS} we can see that 
$$(2,5)\to (7,5);\;  (5,9)\to (7,9); \; (1,8)\to (9,8);\; (4,6)\to (5,6); \; \text{ and }\; (6,10)\to (5,10).$$   

\begin{figure} [htbp]
\begin{center} 
\includegraphics[scale=.8]{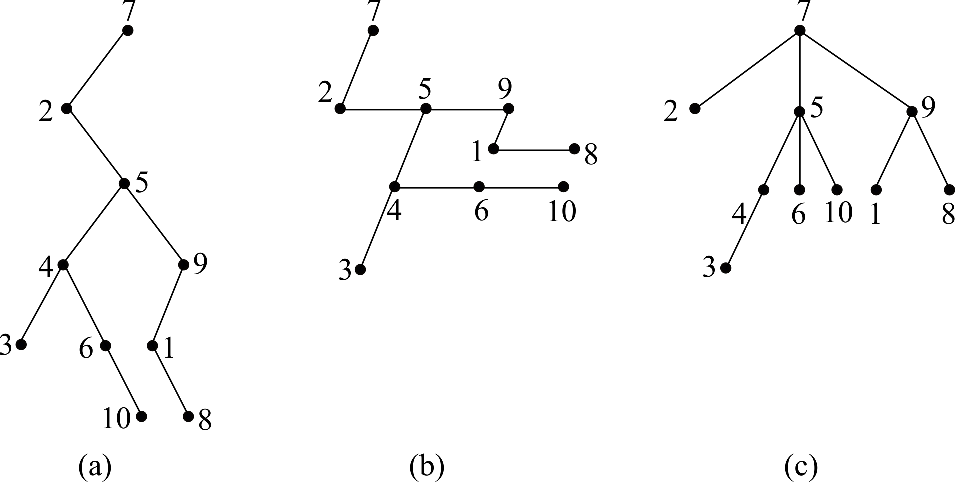} 
\end{center}
\caption{ (a) Left LBS. \hspace{.9cm} (b) Rotation  \hspace{2.5cm} (c) Non-increasing tree  } \quad \label{CorrepondenceLBS}
\end{figure}

\begin{conjecture}\label{conjecture1}
For $0\le k\le n-1$, the number of left LBS trees with $k$ edges of the form $(i+1,i)$ is equal to the number of non-increasing  trees with $k$ edges of the form $(n,i)$.
\end{conjecture}

In Figure \ref{conjecture} we list all 14 left LBS with four vertices. We also give, for each one, the number of edges of the from $(i+1,i)$ and we find these activities: one of activity 3, three of activity 2, 
six of activity 1, and four of activity 0. 

The conjecture could be stated using non-increasing trees instead of LBS trees. The edges $(i+1,i)$ would be only those in which $i$ is the smallest child of $i+1$.  
So, the activity would be the number of vertices $v$ such that their smallest child is $v-1$.

\begin{figure} [htbp]
\begin{center} 
\includegraphics[scale=.8]{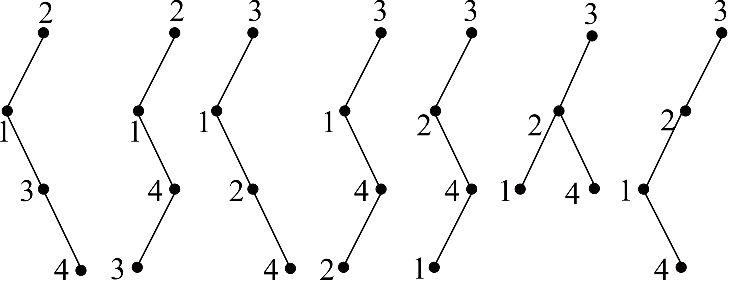}\\[14pt]
\includegraphics[scale=.8]{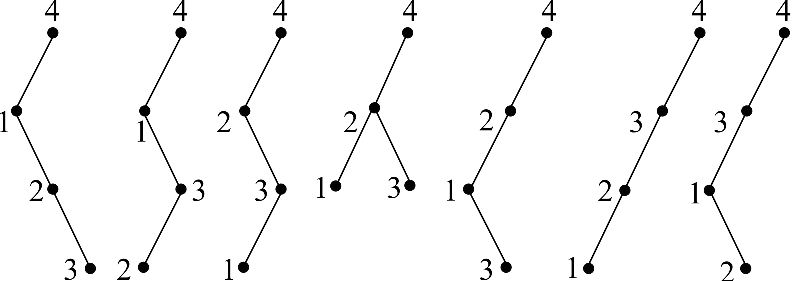}
\end{center}
\caption{ `Loosing' activities. From left to right: \textsf{1}, \textsf{2}, \textsf{0}, \textsf{0}, \textsf{1}, \textsf{2}, 
\textsf{2}, \textsf{0}, \textsf{1}, \textsf{0}, \textsf{1}, \textsf{1}, \textsf{3}, \textsf{1}.}
\label{conjecture}
\end{figure}

\section{Appendix. Algorithms. } \label{Appendix}

\begin{proof}[Proof of Theorem \ref{NumberNBCTrees}]  We prove that Part \eqref{NumberNBCTreesPart1} is equivalent to Part \eqref{NumberNBCTreesPart3}. From Theorems \ref{CFM}  
and \ref{activity1} we obtain a straightforward proof  (because the edges $(n, i)$ define an activity) 
that the number of rooted trees on $n$ vertices where the vertex $1$  has degree $k$ is given by $ {n-1\choose k} (n-1)^{n-1-k}$. 

We give here a second constructive proof of this fact that may lead to a better understanding of this property. 

{\bf The Pr\"uffer code.} The coding algorithm (see Algorithm \ref{TPart1}) takes as an input a rooted labeled tree $T$ on $n$ vertices and gives 
as an output a word $w$ of length $n-1$ on the alphabet $\{1,2, \dots, n\}$.   
In this algorithm  $\lambda$ represents the empty word and $l_i$ represents the smallest leaf of $T$.

\begin{algorithm}[htbp]
\begin{algorithmic}
\Procedure{Pr\"uffer code}{$T$} 
\State $w\gets \lambda$ (the empty word)
    \For{$i = 1 \to n-1$}
        \State $w \gets w l_i$ \Comment{$l_i$= smallest leaf of $T$}
        \State $T \gets T \setminus l_i$  \Comment{delete $l_i$ from $T$}
    \EndFor
 \State \textbf{Return} $w$ 
 \EndProcedure
\end{algorithmic}
\caption{The Pr\"uffer code}
\label{TPart1}
\end{algorithm}

The decoding algorithm (see Algorithm \ref{TPart1deco}) takes as an input a length $n-1$ word on the 
alphabet $\{1,2, \dots, n\}$, and gives as an output a tree. 

Running Algorithm \ref{TPart1deco} we replace a letter $x$ by a pair $(x,i)$. So, the word that contains only letter at the beginning, 
during the process will contain both regular letters and pairs of letters and at the end only pairs. The Algorithm \ref{TPart4}, that is very similar to Algorithm \ref{TPart1deco}, does the same replacement of letters by pairs.

\begin{algorithm}[htbp]
\begin{algorithmic}
\Procedure{Pr\"uffer decoding}{$w$} 
    \For{$i = 1 \to n$} 
   	\If{ $i$ \text{ is the last letter of $w$ } }{  \text{root} $\gets i$}
        		   \Else 
		   \State \text{Search for the last occurrence of $i$ in $w$}
		   \State  \text{$x \gets$ the first regular letter after the last $i$.} 
		   \Comment{(Suppose that the last occurrence of $i$ is in a position $j$. That is, $w_j=i$ and the first regular letter $x$ is  
		   at position $j^{\prime}$, i.e.  $w_j^{\prime}=x$ with $j^{\prime}>j$.) }
		   \State $w_j\gets (x,i)$ 	 
   \EndIf
    \EndFor
    \State $T \gets \text{ root }$ \Comment{$T$ is the tree}
      \While {$w \ne\emptyset$}
			\State Let  $(x,i)$ be the first letter of $\omega$  
   			\State $T\gets T\cup (x,i)$  
			\State $w\gets w\setminus (x,i)$
       \EndWhile
    \State \textbf{Return} $T$ 
 \EndProcedure
\end{algorithmic}
\caption{The Pr\"uffer decoding } \label{TPart1deco}
\end{algorithm}

{\bf Proof that Part \eqref{NumberNBCTreesPart3} is equivalent to Part \eqref{NumberNBCTreesPart2}}. 

{\bf The Blue code.}
The coding algorithm (see Algorithm \ref{TPart3}) takes as an input a rooted labeled tree $T$ on $n$ vertices and gives as an output a word 
$w$ of length $n-1$ on the alphabet $\{b, 1,2, \dots, n-1\}$.  In this algorithm  $\ell_i$ represents the smallest leaf of $T$ and $p_i$ its parent. 

\begin{algorithm}[htbp]
\begin{algorithmic}
\Procedure{Blue coding}{$T$} 
 \State $w \gets \lambda $ (the empty word) 
\While{$T$ has a leaf} \Comment{stops when $T$ has no leaf}
   \For{$i = 1 \to n-1$} 
   	 \If{ $p_i=\ell_i+1$ }{  $w \gets wb$} \Comment{$\ell_i$ is the smallest leaf with parent $p_i$} 
        		\Else 
        			\If{ $p_i>\ell_i$ }{  $w \gets w(p_i-1)$}
        		   		\Else{ $ w \gets wp_i$}
				\State $T \gets T\setminus \ell_i $
    			\EndIf
     	\EndIf
   \EndFor
\EndWhile
 \State \textbf{Return} $w$ 
 \EndProcedure
\end{algorithmic}
\caption{The Blue code for a rooted tree}
\label{TPart3}
\end{algorithm}

The decoding of the word $w$ needs to correct the two transformations of the coding. The procedure is like in our Pr\"uffer decoding to find  
the parents of the vertices from 1 to $n$.
\end{proof}

\begin{algorithm}[htbp]
\begin{algorithmic}
\Procedure{Blue decoding}{$w$} 
\State $\omega \gets w $ 
  \For{$i = 1 \to n$} 
        \If{ $i$ is the last letter of $w$ }{  root $\gets i$}
    		\If{ $i \not \in w$ }{ 0 $\gets j$}
    		\EndIf
        \Else 
                 \State \text{Search the last occurrence of $i$ in $\omega$} 
                 \State  \text{$x \gets$ the first regular letter after $i$}
                 \Comment{(Suppose that last occurrence of $i$ is at position $j$ (so, $\omega_j=i$) and the first regular letter $x$ is at position $j\prime$ 
                 (so, $\omega_{j\prime}=x$) with  $j\prime >j$. )} 
        		\If{ $x=b$ }{ $\omega_{j\prime} \gets ((1+i),i) $ }
        		       \Else 
        				\If{ $x>i$ }{ $w_j \gets  ((x-1),i) $ }
        		   \Else { $\omega_{j\prime} \gets (x,i) $}
    				\EndIf
	\EndIf
     \EndIf
      \EndFor
 \State $T \gets \text{ root }$       
      \While {$\omega \ne\emptyset$}
                        \State Let  $(x,i)$ be the first letter of $\omega$  
   			\State Edges $\gets$ edges $\cup (x,i)$  
   			\State $\omega \gets \omega \setminus (x,i)$  
       \EndWhile
    \State \textbf{Return} $T$  \Comment{defined by roots and edges}
 \EndProcedure
\end{algorithmic}
\caption{The Blue decoding for a word } \label{TPart4}
\end{algorithm}

\textbf{Acknowledgments}.  
The first author was partially supported by The Citadel Foundation.
The second author was partially supported by TEOMATRO project, grant number ANR-10-BLAN 0207. 

We would like to express our special gratitude to Thomas Zaslavsky for his helpful comments and valuable advices.

The authors are grateful to the referees for the helpful suggestions and comments that help improve the paper.

\end{document}